\documentclass[a4paper]{article}

\usepackage[utf8]{inputenc}
\usepackage[T2A]{fontenc}

\usepackage[english]{babel}

\usepackage{amsmath,amssymb,amsthm}
\usepackage{mathrsfs}
\usepackage{authblk}

\usepackage[margin = 2.5cm]{geometry}
\usepackage{nccfoots}
\usepackage[hang,flushmargin]{footmisc} 

\usepackage[pdfauthor={Cherkashin--Pozorov--Teplitskaya},pdftitle={Universal Ahlfors--David regularity of Steiner trees}]{hyperref}

\usepackage{pgf,tikz}

\usetikzlibrary{calc,math,quotes,babel}
\usetikzlibrary{angles,decorations.markings}

\tikzset{mydeco/.style={pic actions/.append code=\tikzset{postaction=decorate}}}

\newtheorem{remark}{Remark}
\newtheorem{proposition}{Proposition}
\newtheorem{lemma}{Lemma}
\newtheorem{corollary}{Corollary}

\newtheorem{theorem}{Theorem}
\newtheorem{problem}{Problem}
\newtheorem{definition}{Definition}

\newtheorem{example}{Example}

\DeclareMathOperator{\dist}{dist}

\DeclareMathOperator{\conv}{conv}
\DeclareMathOperator{\diam}{diam}
\DeclareMathOperator{\Int}{Int}

\newcommand{\St}{\mathcal{S}t}
\renewcommand{\H}{\mathcal{H}^1}

\title{Universal Ahlfors--David regularity of Steiner trees}

\author[1]{Danila Cherkashin}
\author[2]{Pavel Prozorov}
\author[3]{Yana Teplitskaya}

\affil[1]{Institute of Mathematics and Informatics, Bulgarian Academy of Sciences, Sofia}
\affil[2]{Saint Petersburg State University, Saint Petersburg, Russia}
\affil[3]{Laboratoire de Math{\'e}matiques d’Orsay, Universit{\'e} Paris-Saclay, CNRS, Orsay, France}

\begin{document}

\maketitle

\begin{abstract}
    The celebrated Steiner tree problem is the problem of finding a set $\St$ of minimum one-dimensional Hausdorff measure $\H$ (length) such that $\St \cup \mathcal{A}$ is connected, where $\mathcal{A} \subset \mathbb{R}^d$ is a given compact set. 
    Paolini and Stepanov provided very general existence and regularity results for the Steiner problem. 
    Their main regularity result is that under a natural assumption, $\H(\St) < \infty$, for almost every $\varepsilon>0$ the set $\St_\varepsilon := \St\setminus B_\varepsilon(\mathcal A)$ is an embedded finite forest (acyclic graph). 
    
    We give a quantitative regularity result by proving that the set $\St_\varepsilon$ is Ahlfors--David regular with constants that depend only on $d$ (and not on $\mathcal{A}$). Namely, for $d > 2$, every $\varepsilon > 0$, every $x \in \St_\varepsilon$, and every choice of $\rho \in (0,1)$, we have
    \[
    \frac{\H \left (\St_\varepsilon \cap B_{\rho \varepsilon}(x) \right) }{\varepsilon} \leq \left ( \frac{144d}{1-\rho} \right) ^{d-2}.
    \]
    As a corollary, we obtain a density-type result, i.e. that the set $\St_\varepsilon \cap B_{\rho \varepsilon}(x)$ consists of at most
    \[
    \left ( \frac{144d}{1-\rho} \right) ^{d-1}
    \]
    line segments.

    In the plane (i.e., for $d=2$), it is possible to obtain tight structural results. 

\end{abstract}

\textbf{MSC 2020:} 49Q20, 05C63. 

\section{Introduction}

The Steiner minimal tree problem (hereafter simply the Steiner problem) is an umbrella term for an entire class of problems involving the minimization of the length of a set that connects a given collection of points. We will study the Steiner problem in two specific formulations (Problems~\ref{pr:finite} and~\ref{pr:general}).

\begin{problem}[Finite planar Steiner problem] \label{pr:finite}
    For a given finite set of points $\mathcal A = \{ a_1, \dots, a_n \} \subset \mathbb{R}^2$, find a connected set $\St$ of minimum length that contains $\mathcal A$.
\end{problem}

It is known that the set $\St$ exists (but is not always unique) and is a union of a finite set of segments. Thus, $\St$ can be identified with an embedded graph whose vertex set contains $\mathcal A$ and whose edges are straight line segments. This graph $G$ is connected and contains no cycles, which is why the set $\St$ is a tree, explaining its name.
It is known that the vertices of $G$ cannot have degree greater than 3. Furthermore, only vertices corresponding to some $a_i$ can have degree 1 or 2; the remaining vertices have degree 3 and are called branching points (note that a point from $\mathcal A$ can also have degree 3). The number of branching points does not exceed $n-2$. The pairwise angles between edges meeting at any vertex are at least $2\pi/3$. Thus, three edges meet at a branching point with pairwise angles of $2\pi/3$. In any ambient dimension $d$, the three edges adjacent to a branching point lie in a two-dimensional plane.
Branching points outside $\mathcal{A}$ are called Steiner points. If all vertices $a_i$ have degree 1, then the number of Steiner points is $n-2$ (and the converse also holds), and $\St$ is called a full Steiner tree for $\mathcal A$. The proof of these properties of Steiner trees and other information about them can be found in the book~\cite{hwang1992steiner} and in the work~\cite{gilbert1968steiner}.

We call a locally minimal tree for a finite set $\mathcal A$ a connected, compact, acyclic set $S$ containing $\mathcal A$, such that for every point $x \in S$, there exists a neighborhood $U \ni x$ such that $S \cap U$ coincides with a Steiner tree connecting the points $(S \cap \partial U) \cup (\mathcal A\cap U)$.
A locally minimal tree possesses the following properties of a Steiner tree: it consists of a finite set of segments, and the angles between the segments meeting at any point of $S$ are at least $2\pi/3$. 

\subsection{Notation}

For a given set $X \subset \mathbb{R}^d$, we denote by $\overline{X}$ its closure, by $\Int (X)$ its interior, and by $\partial X$ its topological boundary. The size of a finite set $X$ is denoted by $\# X$.
Let $B_\rho(x)$ stand for the open ball of radius $\rho$ centered at a point $x$.

For given points $b$, $c \in \mathbb{R}^d$, we use the notation $[bc]$, $[bc)$ and $(bc)$ for the corresponding closed line segment, ray, and line respectively.
The distance between points $b$ and $c$ is denoted by $\dist(b,c)$, except in the case $d=2$ and a complex coordinate system in which we write $|b-c|$.
Also, we denote the origin by $o$.

\subsection{General Steiner problem}

If $S$ is a subset of a topological space, we say that $S$ is \textit{connected} if it is not the disjoint union of two nonempty closed sets. We say that $S$ is \textit{path-connected} if
for any given two points of $S$ there is a continuous curve joining them in $S$.
We say that $S$ contains no loops if no subset of $S$ is homeomorphic to the circle $\mathbb S^1$. 
We say that $S$ is a topological tree if $S$ is connected and contains no loops.

Let $\mathcal{A}\subset \mathbb R^d$ and $\St\subset \mathbb R^d$. 
We say that a closed set $\St$ is \emph{Steiner set} for $\mathcal{A}$ if $\St\cup \mathcal{A}$ is connected and 
$\H(\St) \le \H(S)$ whenever $S\subset \mathbb R^d$ and $S\cup \mathcal{A}$ is connected.
We say that $\St$ is a \emph{Steiner tree} if $\St$ is a Steiner set and is itself connected.    
A Steiner tree $\St$ is called \textit{full} if $\St \setminus \mathcal{A}$ is connected.

\begin{problem}[General metric Steiner problem] \label{pr:general}
Find Steiner sets for a given compact $\mathcal{A} \subset \mathbb{R}^d$.
\end{problem}

The following theorem summarizes the general results from \cite{paolini2013existence} (in fact, one can replace $\mathbb R^d$ with a proper metric space $\mathcal X$).

\begin{theorem}[Paolini--Stepanov,~2013~\cite{paolini2013existence}] \label{th:PaoSte}
If $\mathcal{A}\subset \mathbb R^d$ is a compact set then there exists a Steiner set 
$\St$ for $\mathcal{A}$. 
Moreover, if $\St$ is a Steiner set for $\mathcal{A}$ and $\H(\St)<+\infty$ then 
\begin{enumerate}
 \item $\St \cup \mathcal A$ is compact;
 \item $\St \setminus \mathcal A$ has at most countably many connected components,
 and each of them has positive length;
 \item $\overline \St$ contains no loops (homeomorphic image of $\mathbb S^1$);
 \item the closure of every connected component of $\St$ is a topological tree
 with endpoints on $\mathcal A$ (in particular, it has at most a countable number of branching points) and has at most one endpoint on each connected component of $\mathcal A$;
    \item if $\mathcal A$ is finite then $\overline \St = \St\cup \mathcal A$ is an embedding of a finite tree;
    \item for almost every $\varepsilon>0$ the set $\St_\varepsilon := \St\setminus B_\varepsilon(\mathcal A)$
    is an embedding of a finite graph.
\end{enumerate}
\end{theorem}

We will frequently use the following corollary.

\begin{corollary} \label{cor:everyeps}
 Let $\St$ be a Steiner tree of finite length for a compact set $\mathcal{A} \subset \mathbb{R}^d$, and assume that the ball $B_r(x)$ does not intersect $\mathcal{A}$. Then for any $\varepsilon > 0$, the set
 \[
 \St \cap \overline{B_{r-\varepsilon}(x)}
 \]
 consists of a finite number of segments.
\end{corollary}

\begin{proof}
By item 6 of Theorem~\ref{th:PaoSte}, for almost every $\varepsilon' > 0$, the set
\[
\St_{\varepsilon'} := \St\setminus B_{\varepsilon'}(\mathcal A)
\]
consists of a finite number of segments. By choosing $\varepsilon' < \varepsilon$ such that this finiteness holds, we obtain that $\St \cap \overline{B_{r-\varepsilon}(x)} \subset \St_{\varepsilon'}$.
Since the ball $\overline{B_{r-\varepsilon}(x)}$ is convex, it intersects each of the segments of $\St_{\varepsilon'}$ in at most one segment.
\end{proof}

\subsection{Contribution}

The main results of this article are quantitative estimates refining the regularity part of Theorem~\ref{th:PaoSte}.
The following theorem gives regularity estimates in the spirit of Ahlfors--David regularity.

\begin{theorem} \label{th:main}
 Fix $d > 2$. Let $\St$ be a Steiner tree for $\mathcal{A}$ with $\H(\St) < \infty$.
 Assume that for some $s > 0$ and some $x \in \mathbb{R}^d$, the open ball $B_s(x)$ has empty intersection with $\mathcal{A}$, and choose any $\rho \in (0,1)$. Then 
 \[
 \frac{\H (\St \cap B_{\rho s}(x) )}{s} \leq \left ( \frac{144d}{1-\rho} \right) ^{d-2}.
 \]
\end{theorem}

It is more or less clear (and we discuss it in the next subsection) that the condition $\rho < 1$ is necessary.

The bound on the length of $\St\cap B_{\rho s}(x)$ in Theorem~\ref{th:main} implies the following density-type corollary controlling the number of segments and branching points.

\begin{corollary} \label{cor:main}
   Under the assumptions of Theorem~\ref{th:main}, the set $\St \cap B_{\rho s}(x)$ consists of at most
    \[
    \left ( \frac{144d}{1-\rho} \right) ^{d-1}
    \]
    line segments. In particular, this means that the number of branching points inside the ball $B_{\rho s}(x)$ is bounded by the same quantity.
\end{corollary}

The reasons for requiring the condition $\rho \in (0,1)$ are more delicate than those for the finiteness of $\H(\St)$.
However, Theorem~\ref{th:example} shows that even in the plane the statement of Corollary~\ref{cor:main} fails for $\rho = 1$.

Finally, in the notation of Paolini and Stepanov~\cite{paolini2013existence} the results can be reformulated as the following (slightly weaker) statement.

\begin{corollary} \label{cor:inNotationofPaoSte}
  Fix $d > 2$. For every $\varepsilon>0$, every $\rho\in(0,1)$, and every $x \in \St_\varepsilon$, we have
  \[
  \H(\St_\varepsilon \cap B_{\rho \varepsilon}(x)) < \varepsilon \left ( \frac{144d}{1-\rho} \right) ^{d-2}.
  \]
  The set $\St_\varepsilon \cap B_{\rho \varepsilon}(x)$ consists of at most 
   \[
    \left ( \frac{144d}{1-\rho} \right) ^{d-1}
   \]
   line segments.
\end{corollary}

\paragraph{Structure of the paper.} In Subsection~\ref{subsec:ADregularity} we introduce the notion of Ahlfors--David regularity and 
explain why the assumption $\H(\St) < + \infty$ and the restriction to a smaller ball are unavoidable. Section~\ref{sec:classics} contains a few classical results we use; despite the fact that some of them seem to be well-known, they are difficult to find in the required form, so we accompany them with proofs. 
In Section~\ref{sec:example} we present a trickier example, showing that even in the plane we need a smaller ball to ensure that the combinatorial structure inside the ball is finite; in Subsection~\ref{subsec:planar} we prove a certain extremal property of this example. 
Subsection~\ref{subsec:mainproofs} contains the proofs of the main results (Theorem~\ref{th:main} and Corollaries~\ref{cor:main} and~\ref{cor:inNotationofPaoSte}). 
Finally, Section~\ref{sec:open} collects open questions.

\subsection{Ahlfors--David Regularity} \label{subsec:ADregularity}

Recall that for every point $x \in \St \setminus \mathcal{A}$, there exists $\varepsilon > 0$ such that the set $\St \cap B_\varepsilon(x)$ is either a segment or a regular tripod. However, this $\varepsilon$ can significantly depend on $x$.

\begin{definition}[Ahlfors--David regularity]
A set $T$ is called Ahlfors--David regular if there exist constants $c,C > 0$
and a radius $r_0$ such that for every positive $r < r_0$ and every $t \in T$, the following estimates hold:
\[
c \leq \frac{\H(T \cap B_r(t))}{r} \leq C.
\]
\end{definition}

Solutions to the Steiner problem for finite $\mathcal{A}$ (i.e., solutions to Problem~\ref{pr:finite}) are Ahlfors--David regular with sharp constants $c = 1$ and $C = 3$.

Solutions to Problem~\ref{pr:general} are somewhat more complex even in the case of a countable compact $\mathcal{A}$ with a single accumulation point. Let us illustrate this property with the following example (see Fig~\ref{pic:1badA}). 

\begin{example} \label{ex:Abad}
   Consider the following explicitly defined set (see Fig.~\ref{pic:1badA})
   \[
    \mathcal{A}_{bad} = \left \{ \left( \frac{a}{4^k}, \frac{b}{4^k} \right); 0 \leq \frac{a}{4^k}, \frac{b}{4^k} \leq \frac{1}{2^k}, a,b,k \in \mathbb{N} \cup \{0\} \right\}.
   \]
   Clearly, $\mathcal{A}_{bad}$ is a certain set of points with rational coordinates in the unit square and $\mathcal{A}_{bad}$ has a unique accumulation point $(0,0)$.
   It turns out that any Steiner set for $\mathcal{A}_{bad}$ has infinite length.
\end{example}

\begin{proof}

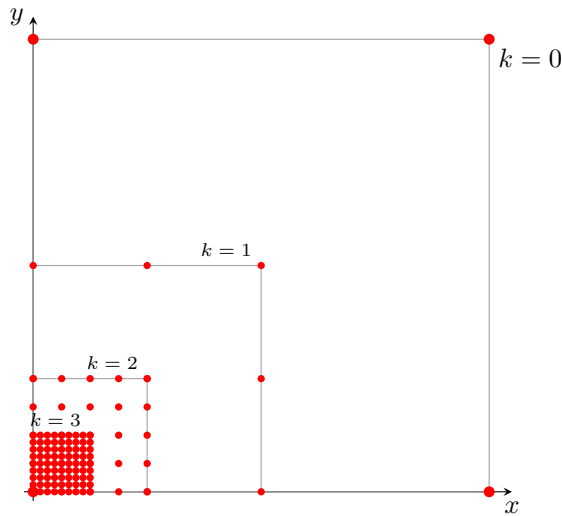
\begin{figure}[h]
    \centering
    \begin{tikzpicture}[scale=6, >=stealth]
  \draw[->] (-0.02,0) -- (1.05,0) node[below] {$x$};
  \draw[->] (0,-0.02) -- (0,1.05) node[left] {$y$};

  \draw[gray!75] (0,0) rectangle (1,1);

  \def\KMAX{3} 

  \foreach \k in {1,...,\KMAX}{
    \pgfmathsetmacro{\step}{1/pow(4,\k)}
    \pgfmathsetmacro{\mx}{1/pow(2,\k)}
    \pgfmathtruncatemacro{\n}{pow(2,\k)}

    \draw[gray!75] (0,0) rectangle (\mx,\mx);
    \node[anchor=south east, font=\scriptsize] at (\mx,\mx) {$k=\k$};

    \foreach \i in {0,...,\n}{
      \foreach \j in {0,...,\n}{
        \fill[red] ({\i*\step},{\j*\step}) circle[radius=0.008];
      }
    }
  }
  \foreach \X/\Y in {0/0, 1/0, 0/1, 1/1} {
    \fill[red] (\X,\Y) circle[radius=0.012];
  }
  \node[below right] at (1,1) {$k=0$};
  
\end{tikzpicture}
    \caption{Example of a compact $\mathcal{A}$ that has a Steiner set of infinite length}
    \label{pic:1badA}
\end{figure}

By construction, the points of $\mathcal{A}_{bad}$ lying in the corner region with vertices $(0, 2^{-k})$, $(0,2^{1-k})$, $(2^{1-k},2^{1-k})$, $(2^{1-k},0)$, $(2^{-k},0)$ and $(2^{-k},2^{-k})$ have neighborhoods of radius $10^{-1}4^{-k}$ that are pairwise disjoint (also disjoint from the corresponding neighborhoods for other values of $k$). Any Steiner tree for $\mathcal{A}_{bad}$ must connect each point to the boundary of its neighborhood, and the corresponding paths do not intersect. This means that the length of any connected set containing $\mathcal{A}_{bad}$ is at least the sum of the neighborhood radii, that is, the length of any Steiner tree is at least
\[
\sum_{k=1}^\infty \frac{1}{10 \cdot 4^k} 16^{k-1} = \infty.
\]
In such a situation, \textit{any} connected set containing $\mathcal{A}_{bad}$ is a Steiner tree, and there is no regularity to discuss.
\end{proof}

Thus, it is necessary to require $\H(\St) < +\infty$. 
However, even in this case, the Steiner tree can be as irregular as $\mathcal{A}$; in particular, it may have an arbitrarily long length in a given ball.
An alternative explanation of this phenomenon is the fact that the length of a Steiner tree for $N$ random points in the unit ball in $\mathbb{R}^d$ grows proportionally to $N^{\frac{d-1}{d}}$ (the maximum length has the same order of growth).

There are two ways to solve this problem: impose additional conditions on $\mathcal{A}$, or modify the definition. The following modification seems natural to us.

\begin{definition}[Ahlfors--David regularity of Steiner trees] \label{def:AD}
A Steiner tree $\St$ is called Ahlfors--David regular if there exist constants $c,C > 0$
such that for every ball $B_r(x)$ with center $x\in\St$ that does not intersect $\mathcal A$ we have
\[
c \leq \frac{\H(\St \cap B_r(x))}{r} \leq C.
\]
\end{definition}

It is worth noting that the intersection $\St \cap B_r(x)$ is not necessarily connected.

It is clear that in any Euclidean space $\mathbb{R}^d$, we can take $c = 2$, and this estimate is sharp for points in the interior of any segment from any Steiner tree and small enough $r$. In the plane, regularity with constant $C = 2\pi$ is obvious (a competitor can be produced by a replacement of a part of a Steiner set inside a circle with the boundary of the circle), but finding the smallest possible value of $C$ is a nontrivial problem.

In higher dimensions it is impossible to have universal regularity in the sense of Definition~\ref{def:AD}.
Indeed, if we take $N$ random points on the sphere $\mathbb{S}^{d-1} = \partial B_1(o)$, then the length of a Steiner tree for them is at most $N$ (and so finite). 
However, the sum of the radii of disjoint neighborhoods as in the analysis of Example~\ref{ex:Abad} will tend to infinity as $N \to \infty$. 
So a constant $C$ in Definition~\ref{def:AD} depends on $\St$. 
On the other hand, if we consider a smaller ball $B_{\rho}(o)$, then most of the length lies outside $B_{\rho}(o)$ and a universal estimate like the one in Theorem~\ref{th:main} becomes possible.

\section{Classical results} \label{sec:classics}

In this section, we present a number of classical results concerning Steiner trees. 
We could not find some of them in the required form in the literature, so we prove suitable variants.

\subsection{Convex hull property}

It is well-known that a Steiner tree for a finite $\mathcal{A}$ lies in $\conv \mathcal{A}$. We need the following generalization.

\begin{lemma} \label{lm:conv}
    A Steiner set $\St$ for a compact $\mathcal{A} \subset \mathbb{R}^d$ is a subset of $\conv \mathcal{A}$, provided that $\H(\St) < \infty$.
\end{lemma}

\begin{proof}
    Set $\mathcal{B} = \conv (\St \cup \mathcal{A})$ and note that $\mathcal{B}$ is a compact set.
    Consider $x \in \St \setminus \mathcal{A}$. For a sufficiently small $\varepsilon > 0$, the point $x$ belongs to the interior of a finite forest $\St_\varepsilon$.
    For a small enough $\delta > 0$ we have $\St_\varepsilon \cap B_\delta(x) = \St \cap B_\delta(x)$ and this intersection is a segment or a regular tripod. Thus $x$ cannot be an extreme point of $\mathcal{B}$. 
    By the Krein--Milman theorem (see Theorem 3.23 in~\cite{RudinFA}), the set $\mathcal{B}$ is the convex hull of its extreme points and thus $\mathcal{B} = \conv \mathcal{A}$. Finally,
    \[
    \St \subset \conv \St \subset \mathcal{B} = \conv \mathcal{A}.
    \]
\end{proof}

\subsection{Combinatorics of finite locally minimal trees}

The \textit{topology} of an embedded forest $F$ is the corresponding abstract graph whose vertex set consists of leaves and branching points of $F$.
We say that the topology $T$ of a tree $S$ is full if $S$ is full.

In this subsection we consider the set of topologies with $n$ terminal points.

The topology is called \textit{degenerate} if a point from $\mathcal{A}$ has degree 3 in it. The name is due to the fact that a random tree on $n$ vertices has a degenerate topology with probability 0, and any other topology occurs with positive probability.
Note that if all vertices lie on a sphere, no locally minimal tree can have a degenerate topology.

Next, we call a topology $T$ \textit{realizable} for $\mathcal{A}$ if there exists a locally minimal tree $S(\mathcal{A})$ with topology $T$; we will denote this tree by $S_T(\mathcal{A})$.

It follows from Proposition~\ref{pr:convex} that $S_T(\mathcal{A})$ is well-defined. Moreover, one can construct $S_T(\mathcal{A})$ (or show that it is impossible) in linear time $O(n)$, where $n$ is the number of terminals, see~\cite{hwang1986linear}. However, we rarely know a priori which realizable topologies correspond to
Steiner trees. Although the number of possible topologies for an $n$-point configuration is finite, checking all of them can take a long time, as this number of topologies grows factorially with $n$, see~\cite{gilbert1968steiner}. Furthermore, the Steiner problem is NP-hard, see~\cite{garey1977complexity}.

For a full topology $T$ with $n$ terminals, we define $D(T)$ as the set of topologies with $n$ terminals that can be obtained from $T$ by contracting some edges (possibly none, i.e., $T \in D(T)$) connecting a terminal to a Steiner point (these edges must have pairwise distinct ends).

\begin{proposition}[Gilbert--Pollak~\cite{gilbert1968steiner}, Hwang--Weng~\cite{hwang1992shortest}] \label{pr:convex}
Let $T$ be a full topology with $n$ terminals and $\mathcal{A}$ be a finite subset of cardinality $n$ in the plane. Consider the function 
\[
L(y_1, \ldots , y_{n-2}) : (\mathbb{R}^2)^{n-2} \to \mathbb{R},
\]
which is the length of the tree on the set of vertices $\mathcal{A} \cup \{y_1, \dots , y_{n-2}\}$ with branching points $y_1, \dots , y_{n-2}$ connected by straight edges according to the topology $T$ (we allow $y_i$ to coincide with each other and with terminals, as well as edge intersections). Then
$L$ has a unique local minimum, and consequently, there is at most one locally minimal tree for $\mathcal{A}$ with a topology from $D(T)$.
\end{proposition}

It is easy to see that the sets $D(T)$ form a partition of the set of all non-degenerate topologies with $n$ terminals.
Indeed, consider a non-full topology $R$ that is not degenerate. For each vertex $x$ of degree two in $R$, we add a new branching point, connect it to $x$'s neighbors (say, $y$ and $z$), and also to $x$ itself, and remove the edges $xy$, $xz$. The resulting full abstract tree $T$ is the unique one for which $R \in D(T)$.

\subsection{Geometry of planar finite locally minimal trees}

In the plane there is a nice geometric argument, introduced by Melzak. 

\begin{lemma}[Melzak's reduction,~\cite{melzak1961problem}] \label{lm:Melzak}
Let $S$ be a full locally minimal tree for a finite $\mathcal{A} \subset \mathbb{R}^2$.
Let $q$ be a branching point adjacent to two pendant vertices of $S$, say $p_1$ and $p_2$.
Then the length $\H(S)$ is equal to the length of the connected set $S'$ obtained from $S$ 
by removing the points $p_1$ and $p_2$ and the edges $[p_1q]$ and $[p_2q]$ and prolonging the remaining edge arriving 
in $q$ up to a new vertex $p$ such that $p p_1 p_2$ is an equilateral triangle and (among the two possibilities) $p$ and $q$ are on opposite sides of the line  $(p_1,p_2)$.
\end{lemma}

\begin{proof}
    Clearly, $\angle p_1pp_2 = \pi/3$, $\angle p_1qp_2 = 2\pi/3$ and the line $(p_1,p_2)$ separates $p$ and $q$. Hence $p,p_1,p_2$ and $q$ are cocircular. Then $\angle p_1qp = \angle p_1p_2p = \pi/3$, so $[pq]$ is a continuation of the remaining edge adjacent to $q$ in $S$.
By Ptolemy's theorem 
\[
|pq| \cdot |p_1p_2| = |p_1q| \cdot |pp_2| + |p_2q| \cdot |pp_1|,
\]
so $|pq| = |p_1q| + |p_2q|$ and $\H(S') = \H(S)$.    
\end{proof}

In higher dimensions there is no such beautiful algorithm, see the discussion of obstacles in~\cite{smith1992find}.

\subsection{Stability of finite solutions}

\begin{lemma} \label{lm:stab}
 Let $\St$ be a minimal Steiner tree for an $n$-point set $\mathcal{A}$ with topology $R \in D(T)$. 
 Assume that any locally minimal tree for $\mathcal{A}$ with a topology not in $D(T)$ has length at least $\H(\St) + \varepsilon$.
 Then for any shift of each terminal by a distance less than $\varepsilon/(5n)$, any Steiner tree for the shifted set of points will have a topology from $D(T)$. 
\end{lemma}

\begin{proof}

Let $\mathcal{A} = \{a_1, \dots, a_n\}$, and let the shifted vertices form the set $\mathcal{B} = \{b_1, \dots, b_n\}$, where $b_i \in B_{\varepsilon/(5n)}(a_i)$.
    
Add segments $[a_ib_i]$ to $\St$, where $b_i$ is the shifted vertex; then replace all edge-segments $[a_ix]$ with $[b_ix]$, and consider the unique local minimum $\St_{new}$ for the terminal set $\mathcal{B}$ from Proposition~\ref{pr:convex}.
Each terminal has degree 1 or 2 in $R$, so by the triangle inequality 
\[
\H (\St_{new}) < \H (\St) + 2n \cdot \varepsilon/(5n) = \H (\St) + 2\varepsilon/5.
\]

Consider an arbitrary locally minimal tree $S$ for $\mathcal{B}$ with a topology not in $D(T)$.
Similarly to the previous paragraph, add segments $[a_ib_i]$ to $S$, replace all edge-segments $[b_ix]$ with $[a_ix]$ and consider the unique local minimum $S_{old}$ for the terminal set $\mathcal{A}$ from Proposition~\ref{pr:convex}. Similarly, each terminal has degree 1, 2 or 3, so
    \[
    \H (S_{old}) <  \H (S) + 3n \cdot \varepsilon/(5n) = \H (S) + 3\varepsilon/5.
    \]
By assumption $\H(\St) + \varepsilon < \H (S_{old})$; putting all together we have the inequality
\[
\H (\St_{new}) < \H(\St) + 2\varepsilon/5 < \H (S_{old}) - 3\varepsilon/5 < \H(S),
\] 
which finishes the proof.
\end{proof}

\subsection{Length formula for a planar tree} \label{subsec:maxwell}

In the following lemma, we identify the Euclidean plane with the complex plane $\mathbb{C}$.

\begin{lemma}[Maxwell-type formula] \label{lm:Maxwell}
Let $\St$ be a full locally minimal tree with terminals $p_1,p_2,\dots, p_n \in \mathbb C$. 
Suppose $n>1$ and define $c_k$ as the unit complex number representing the outer direction of 
the unique edge incident to $p_k$.  

Then one can write the length of the tree as
\begin{equation}\label{eq:maxwell}
\H (\St) = \sum_{k=1}^n \bar{c_k} p_k.
\end{equation}
In particular, the right-hand side of~\eqref{eq:maxwell} is a real number.
\end{lemma}

\begin{proof}[Proof of Lemma~\ref{lm:Maxwell}]
Let $v_1,\ldots,v_{2n-2} \in \mathbb{C}$ be the vertices of $\St$ in the complex plane (they include the terminal points $v_k=p_k$ for $k=1,\dots, n$ and $n-2$ branching points) and let $E$ be the set of edges of $\St$.
Let $u(z) := \frac{z}{\lvert z\rvert} = e^{i\arg z}$ be the unit complex number representing the direction of the complex number $z$
so that $\lvert z\rvert = z\cdot \overline{u(z)}$.
One has
\begin{align*}
\H (\St) 
&= \sum_{(k,j) \in E} |v_k - v_j| 
= \sum_{(k,j) \in E} (v_k - v_j) \cdot \overline{u(v_k - v_j)}\\
&= \sum_{(k,j)\in E} \left ( v_k \cdot \overline{u(v_k-v_j)} + v_j \cdot \overline{u(v_j-v_k)} \right )
 = \sum_{k=1}^{2n-2} v_k \cdot \sum_{(k,j)\in E} \overline{u(v_k-v_j)}.
\end{align*}
If $v_k$ is a branching point, then the corresponding sum of directions is zero, so we have only a sum over the terminal points.
\end{proof}

The following corollary follows from the fact that the length does not depend on translation.

\begin{corollary} \label{cor:windrose}
 One has $\sum_{k=1}^n \bar{c_k} = 0$.
\end{corollary}

Another corollary concerns a cocircular $\mathcal{A}$.

\begin{corollary} \label{cor:maxwellCocircular}
Let $\mathcal{A} = \{a_1,\dots, a_n\}$ be a finite subset of a unit circle $C$. Suppose that $\St$ is a full Steiner tree for $\mathcal{A}$. Then
\[
\H(\St) = \sum_{i=1}^n \cos \beta_i,
\]
where $\beta_i$ is the angle between the segment of $\St$ adjacent to $a_i$ and the radius $[oa_i]$. 
\end{corollary}

\begin{proof}
    By Lemma~\ref{lm:Maxwell} 
    \[
    \H(\St) = \sum_{i=1}^n e^{\pm \beta_i},
    \]
    where a sign depends on which of the two possible rays forming an angle $\beta_i$ with the radius $[oa_i]$ contains the segment of $\St$. Since $\H(\St)$ is a real number, we have
    \[
    \H(\St) = \Re \sum_{i=1}^n e^{\pm \beta_i} = \sum_{i=1}^n \Re e^{\pm \beta_i} = \sum_{i=1}^n \cos \beta_i.
    \]
\end{proof}

Next, the stability of solutions can be combined with the Maxwell-type formula as follows. 
Let $v_1, \dots, v_n$ be points in the plane for which the Steiner tree $\St_v$ with topology $T$ is unique and full. 
Let $u_1, \dots, u_n$ be the points obtained by shifting $v_i$, such that $|u_j - v_j| < \varepsilon$ for each $j$.
By Lemma~\ref{lm:stab}, for sufficiently small $\varepsilon$, the Steiner tree for the new points $\St_u$ preserves the topology from $D(T)$. 
For an even smaller $\varepsilon$, the tree $\St_u$ has topology $T$, because the possibility of drawing a full tree with a given topology is determined by a finite system of strict inequalities. 

Then by the Maxwell-type formula, the lengths of the Steiner trees are expressed as
\[
\H(\St_v) = \sum_{j=1}^n \bar{c_j} v_j, \quad \quad \H(\St_u) = e^{i\alpha} \sum_{j=1}^n \bar{c_j} u_j,
\]
where $\alpha$ is the angle between the corresponding segments of the trees $\St_v$ and $\St_u$.

Accordingly, if
\[
\sum_{j=1}^n \bar{c_j} (u_j - v_j) \in \mathbb{R},
\]
then the corresponding tree segments are parallel. Moreover, if 
\[
\sum_{j=1}^n \bar{c_j} (u_j - v_j) = 0,
\]
then additionally $\H(\St_v) = \H(\St_u)$.

\subsection{Minimum spanning trees. Steiner ratio}

A reasonably good approximation to the NP-hard Steiner problem is the Minimum Spanning Tree (MST) problem, which can be solved in polynomial time. There are many polynomial-time algorithms, as well as heuristics for special classes of graphs. Prim's algorithm (first described by Jarn{\' i}k) will suffice for us. 

First, we choose an arbitrary vertex and an incident edge of minimum weight. The found edge and the two vertices it connects form a tree. Consider all edges with one endpoint in the current tree and the other outside it; from these edges, the edge of the smallest weight is chosen. The edge selected at each step is attached to the tree. The tree continues to grow until all vertices of the original graph are included.

The Steiner ratio of a metric space $X$ is defined as the supremum of the ratio between the length of a minimum spanning tree and that of a Steiner minimal tree over finite sets $\mathcal{A} \subset X$.

\begin{proposition}[Graham--Hwang, 1976~\cite{graham1976remarks}] \label{pr:ratio}
For any Euclidean space $\mathbb{R}^d$, the Steiner ratio is at most $\sqrt{3}$.
\end{proposition}

It is worth noting that a direct application of the triangle inequality gives that the Steiner ratio of an arbitrary metric space is at most two, see~\cite{cieslik2013steiner}.

\subsection{Steiner tree on a surface}

\begin{theorem} \label{th:tpoints}
Let $d > 2$. A Steiner tree for arbitrary $t$ points on the sphere $\mathbb{S}^{d-1}$ has length at most
\[
c_d t^\frac{d-2}{d-1}.
\]
Moreover, $c_d = 2+o_d(1)$ and $c_d \leq c_3 = 11.78\dots$ for every $d$.
\end{theorem}

\begin{proof}
    Consider arbitrary points $a_1,\dots, a_t$ on the sphere $\mathbb{S}^{d-1}$; let $\St$ be a Steiner tree for $\{a_i\}$.
    Let $\varepsilon = ct^{-\frac{1}{d-1}}$, where we will choose the constant $c$ later. 
    Let the angle $\varphi$ be such that $\sin \varphi = \varepsilon$ and put $\delta = 2\sin (\varphi/2)$.
    This gives that a spherical cap $B_{\delta} (x) \cap \mathbb{S}^{d-1}$, where $x \in \mathbb{S}^{d-1}$ has angular radius $\varphi$, i.e. consists of all points $x' \in \mathbb{S}^{d-1}$ such that $\angle xox' < \varphi$.
    
    Consider a maximal (with respect to inclusion) set of points $x_1,\dots, x_k \in \mathbb{S}^{d-1}$ for which the spherical caps 
    $B_{\delta} (x_i) \cap \mathbb{S}^{d-1}$ are pairwise disjoint. 
    Then the union of the spherical caps $B_{2 \sin \varphi} (x_i) \cap \mathbb{S}^{d-1}$ covers the sphere $\mathbb{S}^{d-1}$.
    Connect the points $x_1,\dots,x_k$ with a spanning tree, and then each point $a_i$ with the closest of the points $\{x_j\}$.
    Prim's algorithm adds an edge no longer than $2 \sin 2 \varphi$ at each iteration, so
    \[
    \H(\St) \leq \H(MST) +  2 t\sin \varphi \leq 2(k-1)\sin 2\varphi + 2 t\sin \varphi < 2(2k+t)\varepsilon,
    \]
    since $\sin 2\varphi < 2\sin \varphi$.
    Estimate $k$; a direct volume estimate gives
    \[
    k \leq \frac{\mathcal{H}^{d-1}(\mathbb{S}^{d-1})}{\mathcal{H}^{d-1}(B_{\delta} (x_i) \cap \mathbb{S}^{d-1})} \leq \frac{\sqrt{2\pi d}}{\sin^{d-1}\varphi} = \frac{\sqrt{2\pi d}}{\varepsilon^{d-1}},
    \]
    where the last inequality is proved in~\cite{boroczky2003covering}.
    
    Expressing everything in terms of $t$ and $c$, we obtain
    \[
    \H(\St) \leq 2 \left( 2\sqrt{2\pi d} \cdot c^{2-d} + c \right) t^\frac{d-2}{d-1}.
    \]
    Optimizing over $c$, we get the following upper bound, which is obtained by the substitution of $c_{min} = (2\sqrt{2\pi d}(d-2))^{1/(d-1)}$
    \[
    \H(\St) \leq 2 \left (2\sqrt{2\pi d}(d-2) \right)^{1/(d-1)} \frac{d-1}{d-2} t^\frac{d-2}{d-1} = \left(2+o_d(1) \right) t^\frac{d-2}{d-1}.
    \]    
    A numerical analysis of small cases finishes the proof.
\end{proof}

The estimate for $k$ can be improved using classical results on spherical codes by Kabatiansky and Levenshtein~\cite{kabatiansky1978bounds} and Sidel'nikov~\cite{sidel1974new}.

\subsection{Two bounds on length of a set}

The following bound is based on the fact that a subset of the optimal structure inherits certain optimality properties.

\begin{lemma} \label{lm:lengthbound}
Let $\St$ be a Steiner set for $\mathcal{A}$ and $\H(\St) < +\infty$. 
Suppose that $X \subset \St$ and $Y \subset \mathbb{R}^d$ are such that
\[
\St \cup \mathcal A \cup Y \setminus X
\]
is connected. Then $\H(X) \leq \H(Y)$.
\end{lemma}

\begin{proof}
    Consider a competitor $\St' = \St \cup \mathcal A \cup Y \setminus X$. It is connected and 
    \[
    \H(\St') = \H(\St) - \H(X \setminus Y) + \H(Y \setminus X) = \H(\St) - \H(X) + \H(Y).
    \]
    Since $\St$ is a Steiner set, $\H(\St) \leq \H(\St')$ and therefore $\H(X) \leq \H(Y)$.
\end{proof}

The following coarea inequality can be obtained from Theorem 2.1 in~\cite{paolini2013existence} by substituting $X = \mathbb{R}^d$ and $f (\cdot) = \dist(\cdot,x)$ (which is 1-Lipschitz by the triangle inequality).

\begin{lemma}[Coarea inequality] \label{lm:coarea}
Let $S$ be a nonempty measurable subset of $\mathbb{R}^d$ and $x \in \mathbb{R}^d$ be an arbitrary point. Then
\[
\H(S) \geq \int_{r=0}^\infty \# (\partial B_r(x) \cap S) d r.
\]
\end{lemma}

\subsection{Cocircular terminal sets}

\begin{theorem}[Rubinstein--Thomas~\cite{rubinstein1992graham}, 1992] \label{th:cocircular}
    Suppose that a given set $P = \{p_1, p_2, \dots, p_n\}$ lies on a circle $C \subset \mathbb{R}^2$ of radius $r$ with a given cyclic order.
    If at most one of the segments $[p_ip_{i+1}]$, $i = 1,\dots, n$ (here $n+1 = 1$) has length strictly greater than $r$, then a Steiner tree for $P$ has no Steiner points.
\end{theorem} 

We need to adapt the result for the case of infinite $\mathcal{A}$.

\begin{corollary} \label{cor:5pointsaremany}
    Let $C \subset \mathbb{R}^2$ be a circle of radius $r$.
    Let $\mathcal{A} \subset C$ be an infinite compact set such that at most one connected component (arc) of $C \setminus \mathcal{A}$ has diameter greater than $r - \delta$ for some $\delta > 0$. Then a solution $\St$ to the Steiner problem for $\mathcal{A}$ has no Steiner points.
\end{corollary}

\begin{proof}
Assume the contrary and consider a full Steiner tree $\St$ for $\mathcal{A} \subset C$.
Choose a $\delta/4$-network $x_1, \dots, x_m \in \mathcal{A}$ for $\mathcal{A}$. Note that at most one connected component of $C\setminus \{x_1,\dots,x_m\}$ has diameter greater than 
$(r - \delta) + 2\delta/4 = r - \delta/2$. 

Fix some point $x_0 \in \St \cap B_{r-\delta}(o)$, where $o$ is the center of $C$ (if it does not exist, we may consider a smaller $\delta$). 
Since $\St \setminus \mathcal{A}$ is connected and has a finite length, it is path-connected (see for example~\cite{EiSaHa}). For every $1 \leq i \leq m$ consider the first intersection $y_i$ of the path from $x_i$ to $x_0$ with $\partial B_{\delta/8}(x_i)$, and denote by $P_i$ the path from $x_i$ to $y_i$.
By continuity, for every $\varepsilon < \dist(y_i,C)$ the circle $\partial B_{r-\varepsilon}(o)$ intersects $P_i$.
Now consider an $\varepsilon_0$ for which all $y_i$ belong to the same closure $\St_0$ of a connected component of the set $\St \cap B_{r-\varepsilon_0}(o)$.
Since $\St$ is a full tree, $\St_0$ is also a full tree. Also, $\St_0$ contains a terminal point in every $B_{\delta/8}(x_i)$, so $\mathcal{A} \subset B_{3\delta/8}(\mathcal{A}_0)$, where 
$\mathcal{A}_0$ is the terminal set of $\St_0$. Hence at most one connected component of $\partial B_{r-\varepsilon_0}(o) \setminus \mathcal{A}_0$ has diameter greater than 
$r - \delta + 2 \cdot 3 \delta / 8 = r - \delta/4 < r - \varepsilon_0$.

Thus $\St_0$ satisfies the conditions of Theorem~\ref{th:cocircular}, so it cannot be Steiner tree for $\mathcal{A}_0$. Therefore one can replace in $\St$ the part $\St_0$ with a Steiner tree for $\mathcal{A}_0$, which contradicts Lemma~\ref{lm:lengthbound}.
\end{proof}

\begin{corollary}\label{cor:cocircularlength}
   Assume that an open ball $B_r(x)$ does not intersect the terminal set $\mathcal{A} \subset \mathbb{R}^2$ of a Steiner tree $\St$. Then
   \begin{itemize}
       \item[(i)]  the length of $\St \cap B_r(x)$ is at most $2\pi r$;
       \item[(ii)] the length of the union $F$ of connected components of $\St \cap B_r(x)$ which are not segments (i.e., that have at least one branching point) is at most $4\pi r/3 + r = 5.188 \ldots \cdot r$.  
   \end{itemize}
\end{corollary}

\begin{proof}
    Item~(i) follows directly from Lemma~\ref{lm:lengthbound} with $X = \St \cap B_r(x)$ and $Y = \partial B_r(x)$.

    Now consider the closure of a connected component $S$ of $\St \cap B_r(x)$ that has a branching point. By Corollary~\ref{cor:5pointsaremany} the set $\partial B_r(x) \setminus S$ has at least two connected components (arcs) of diameter at least $r$; fix two of them, say $R_1 = \breve{yz}$ and $R_2$. We claim that the set $\partial B_r(x) \setminus \overline{F}$ has a connected component (arc) of diameter at least $r$ inside both $R_1$ and $R_2$. Indeed, every closure of a connected component $S'$, whose terminals belong to $R_1$ has a large connected component (arc) of $\partial B_r(x) \setminus S'$ inside the arc $R_1$; replace $R_1$ with this arc. Since every such component has diameter (and thus length) at least $r$, by item~(i) we have at most 6 such components, so this reduction process is finite. 

    Now apply Lemma~\ref{lm:lengthbound} with $X = F$ and $Y = (\partial B_r(x) \setminus R_1 \setminus R_2) \cup [yz]$.
    Clearly $\H(Y)$ is maximal when $\H([yz])$ is minimal. If $|yz| = r$, then $\H(Y) = 4\pi r/3 + r$.
\end{proof}

\section{Example of a non-regular tree} \label{sec:example}

In this section, we construct an example of a full Steiner tree with infinitely many branching points in a disk, whose terminals lie on the unit circle.

\begin{theorem} \label{th:example}
    There is a full tree $\St$ for a countable terminal set $\mathcal{A} \subset \mathbb{R}^2$ lying on a circle such that the set of branching points of $\St$ has 4 accumulation points. 
\end{theorem}

\begin{remark}
    Actually, these accumulation points are exactly the accumulation points of $\mathcal{A}$. Theorem~\ref{th:4pointsmax} shows that the quantity 4 cannot be increased.
\end{remark}

\begin{proof}[Proof of Theorem~\ref{th:example}]
Here we again identify points in the plane with complex numbers.
Consider the unit circle $\omega$. 
Let $\St_0$ be the Steiner tree for the set $\mathcal{A}_0 = \{\pm e^{\pi i/6}, \pm e^{5\pi i/6}\}$, whose points form rectangle $P$, see the leftmost part of Fig.~\ref{pic:2step0}.
It is easy to see that $\St_0$ makes angles of $\pi/3$ with $\omega$ at the vertices.
We construct a sequence of full trees $\St_j$ for terminals $\mathcal{A}_j$, whose edges are parallel to the same three straight lines.
Moreover, every tree $\St_j$ is unique Steiner tree for $\mathcal{A}_j$, and any other locally minimal tree for $\mathcal{A}_j$ is longer than $\St_j$ by at least $\delta_j > 0$, which will be defined later.

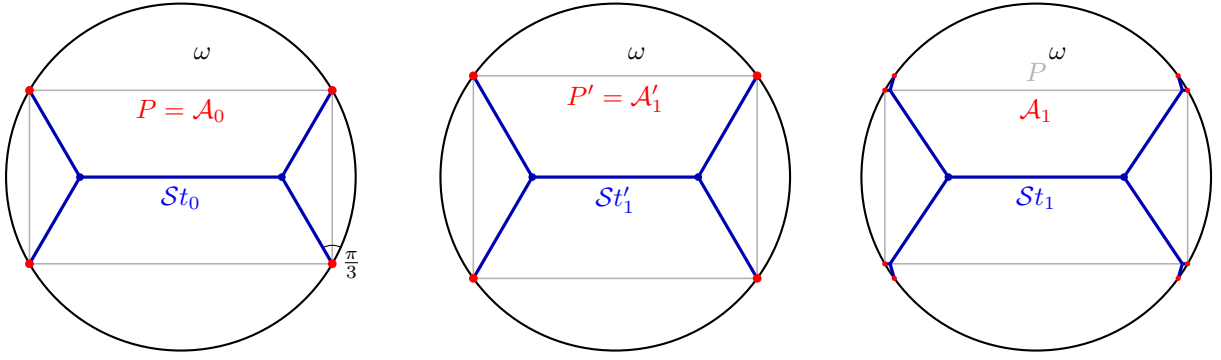
\begin{figure}[h]
    \centering
    \begin{tikzpicture}[scale=2.3]

  \pgfmathsetmacro{\rtthree}{1.7320508075688772} 
  \pgfmathsetmacro{\xC}{\rtthree/2}  
  \pgfmathsetmacro{\yC}{0.5}         
  \pgfmathsetmacro{\xs}{1/sqrt(3)}   

  \draw[thick] (0,0) circle (1);
  \node[below right] at (0.02,0.8) {$\omega$};

  \coordinate (UR) at  (\xC,\yC);     
  \coordinate (UL) at  (-\xC,\yC);    
  \coordinate (LL) at  (-\xC,-\yC);   
  \coordinate (LR) at  (\xC,-\yC);    

  \draw[gray!60, line width=0.6pt] (UL)--(UR)--(LR)--(LL)--cycle;
  \node[red, below] at (0,\yC) {$P=\mathcal{A}_0$};

  \coordinate (SL) at (-\xs,0);
  \coordinate (SR) at (\xs,0);

  \draw[blue!70!black, very thick] (SL)--(SR);         
  \draw[blue!70!black, very thick] (SL)--(UL);
  \draw[blue!70!black, very thick] (SL)--(LL);
  \draw[blue!70!black, very thick] (SR)--(UR);
  \draw[blue!70!black, very thick] (SR)--(LR);
  \node[blue!70!black, below] at (0,0) {$\St_0$};

  \foreach \P in {UL,UR,LL,LR} \fill[red] (\P) circle (0.025);
  \foreach \P in {SL,SR} \fill[blue!70!black] (\P) circle (0.022);


  \pgfmathsetmacro{\eps}{0.25}



  \coordinate (LRt) at ($(LR)+\eps*(\yC,\xC)$);  
  \pic [draw, below right, "$\frac{\pi}{3}$", angle eccentricity=1.2, angle radius=7pt]
      {angle = LRt--LR--SR};


\end{tikzpicture}
\hfill
\begin{tikzpicture}[scale=2.3]

  \pgfmathsetmacro{\eps}{0.1}                 
  \pgfmathsetmacro{\thet}{pi/6 + \eps}        

  \draw[thick] (0,0) circle (1);
  \node[below right] at (0.02,0.8) {$\omega$};

  \pgfmathsetmacro{\cx}{cos(\thet r)}
  \pgfmathsetmacro{\sy}{sin(\thet r)}
  \coordinate (UR) at (\cx,\sy);
  \coordinate (UL) at (-\cx,\sy);
  \coordinate (LL) at (-\cx,-\sy);
  \coordinate (LR) at (\cx,-\sy);

  \draw[gray!60, line width=0.6pt] (UL)--(UR)--(LR)--(LL)--cycle;
  \node[red, below] at (0,\sy) {$P'=\mathcal{A}_1'$};

  \pgfmathsetmacro{\s}{\cx - \sy/sqrt(3)}   
  \coordinate (SL) at (-\s,0);
  \coordinate (SR) at (\s,0);

  \draw[blue!70!black, very thick] (SL)--(SR);   
  \draw[blue!70!black, very thick] (SL)--(UL);
  \draw[blue!70!black, very thick] (SL)--(LL);
  \draw[blue!70!black, very thick] (SR)--(UR);
  \draw[blue!70!black, very thick] (SR)--(LR);
  \node[blue!70!black, below] at (0,0) {$\St_1'$};

  \foreach \P in {UL,UR,LL,LR} \fill[red] (\P) circle (0.025);
  \foreach \P in {SL,SR} \fill[blue!70!black] (\P) circle (0.022);

\end{tikzpicture}
\hfill
\begin{tikzpicture}[scale=2.3]

  \pgfmathsetmacro{\rtthree}{1.7320508075688772} 
  \pgfmathsetmacro{\xC}{\rtthree/2}  
  \pgfmathsetmacro{\yC}{0.5}         
  \pgfmathsetmacro{\xs}{1/sqrt(3)-0.075}   

  \draw[thick] (0,0) circle (1);
  \node[below right] at (0.02,0.8) {$\omega$};

  \coordinate (UR) at  (\xC,\yC);     
  \coordinate (UL) at  (-\xC,\yC);    
  \coordinate (LL) at  (-\xC,-\yC);   
  \coordinate (LR) at  (\xC,-\yC);    

  \pgfmathsetmacro{\eps}{0.1}                 
  \pgfmathsetmacro{\deltax}{0.03}             
  \pgfmathsetmacro{\deltay}{0.00}             
  \pgfmathsetmacro{\thet}{pi/6 + \eps}        

  \pgfmathsetmacro{\cx}{cos(\thet r)}
  \pgfmathsetmacro{\sy}{sin(\thet r)}
  \coordinate (URr) at (\cx,\sy);
  \coordinate (ULl) at (-\cx,\sy);
  \coordinate (LLl) at (-\cx,-\sy);
  \coordinate (LRr) at (\cx,-\sy);

  \draw[gray!60, line width=0.6pt] (UL)--(UR)--(LR)--(LL)--cycle;
  \node[red, below] at (0,\yC) {$\mathcal{A}_1$};
  \node[gray!60, above] at (0,\yC) {$P$};
  
  \coordinate (SL) at (-\xs,0);
  \coordinate (SR) at (\xs,0);


  \coordinate (SLu) at (-\xC+\deltax,\yC-\deltay);
  \coordinate (SLd) at (-\xC+\deltax,-\yC+\deltay);
  \coordinate (SRu) at (\xC-\deltax,\yC-\deltay);
  \coordinate (SRd) at (\xC-\deltax,-\yC+\deltay);

  \draw[blue!70!black, very thick] (SL)--(SR);         
  \draw[blue!70!black, very thick] (SL)--(SLu);
  \draw[blue!70!black, very thick] (SL)--(SLd);
  \draw[blue!70!black, very thick] (SR)--(SRu);
  \draw[blue!70!black, very thick] (SR)--(SRd);
  \draw[blue!70!black, very thick] (SLu)--(ULl);
  \draw[blue!70!black, very thick] (SLu)--(UL);
  \draw[blue!70!black, very thick] (SLd)--(LLl);
  \draw[blue!70!black, very thick] (SLd)--(LL);
  \draw[blue!70!black, very thick] (SRu)--(URr);
  \draw[blue!70!black, very thick] (SRu)--(UR);
  \draw[blue!70!black, very thick] (SRd)--(LRr);
  \draw[blue!70!black, very thick] (SRd)--(LR);
  \node[blue!70!black, below] at (0,0) {$\St_1$};

  \foreach \P in {UL,UR,LL,LR,ULl,URr,LLl,LRr} \fill[red] (\P) circle (0.015);
  \foreach \P in {SL,SR} \fill[blue!70!black] (\P) circle (0.022);

\end{tikzpicture}
    \caption{Transition from $\St_0$ to $\St_1$}
    \label{pic:2step0} 
\end{figure}

At the $j$-th iteration, we shift the terminals from $\mathcal{A}_0$ in $\mathcal{A}_{j-1}$ (by construction $\mathcal{A}_0 \subset \mathcal{A}_{j-1}$ for every $j$) along the circle $\omega$ by an amount $\varepsilon_j < \delta_{j-1}^2/128j$ inside the angle of size $2\pi/3$, obtaining a new rectangle with the vertex set $P_j$. 
Formally, the direction can be defined by the decreasing of $(-1)^j |p_y|$, where $p_y$ is the $y$-coordinate of the point $p \in \mathcal{A}_0$, so the directions of the shifts differ between even and odd iterations.  
Let $\St'_j$ be the Steiner tree for $\mathcal{A}'_j := \mathcal{A}_{j-1} \cup P_j \setminus \mathcal{A}_0$.
By Lemma~\ref{lm:stab}, the tree $\St'_j$ is a unique Steiner set for the terminal set $\mathcal{A}'_j$ and has the same topology as $\St_{j-1}$.

Analogously to the reasoning from the end of Subsection~\ref{subsec:maxwell}, we will show that the tree $\St'_j$ is parallel to $\St_{j-1}$.
Indeed, let $\{c_i\}$ be the set of coefficients in Maxwell-type formula~\eqref{eq:maxwell} for $\St_{j-1}$. 
Write the difference of~\eqref{eq:maxwell} for $\St_{j-1}$ and the same expression with the same coefficients for the terminal set $\mathcal{A}'_j$
\[
\Delta_j := \sum_{u \in \mathcal{A}_{j-1}} \bar{c_u} u - \sum_{u \in \mathcal{A}'_j} \bar{c_u} u =
\sum_{u \in \mathcal{A}_0} \bar{c_u} u - \sum_{u \in P_j} \bar{c_u} u.
\]
In the case of odd $j$ the coefficients $\bar{c_u}$ for $e^{\pi i/6}$, $e^{5\pi i/6}$, $e^{7\pi i/6}$ and $e^{11\pi i/6}$ are $e^{-\pi i/3}$, $e^{-2\pi i/3}$, $e^{2\pi i/3}$ and $e^{\pi i/3}$, respectively. 
In this case 
\[
\Delta_j = -2e^{-i\pi/3} v_j - 2e^{i\pi/3} \bar{v_j},
\]
where $v_j$ is the displacement vector of the vertex $e^{\pi i/6} \in \mathcal{A}_0$.    
Clearly $\Delta_j$ is real, so the trees $\St'_j$ and $\St_{j-1}$ are segment-wise parallel.
The case of even $j$ is completely analogous with 
\[
\Delta_j = -2v_j - 2\bar{v_j}.
\]

Now, for each $x \in \mathcal{A}_0$ denote by $x_j$ the corresponding shifted terminal in $\St'_j$ (so for $x = e^{\pi i/6}$ we have $x_j = x + v_j$). To finalize the $j$-th step we replace for every $x \in \mathcal{A}_0$ the edge $[x_jy(x)_j]$ in $\St'_j$ with a regular tripod with ends $x, y(x)_j, z(x)_j$ and branching point $t(x)_j$, where the point $z(x)_j \in \omega$ is chosen so that the tripod intersects $[x_jy(x)_j]$ along a segment. 
Let $\St_j$ be the resulting tree on terminals $\mathcal{A}_j$ with topology $T_j$; so $\mathcal{A}_j$ is the union of $\mathcal{A}_{j-1}$ and four points $z(x)_j$.

\begin{figure}[h]
    \centering
    \hfill
\begin{tikzpicture}[scale=2.6]
  \usetikzlibrary{calc,angles,quotes}
  \draw[thick] (-0.6,0) -- (1.2,0);
  \node[above] at (-0.5,0.02) {$\omega$};

  \coordinate (X) at (0.35,0);
  
  \coordinate (D)  at ($(X)+0.75*(-0.5,-0.8660254)$); 
  \draw[blue!70!black,ultra thick] (X) -- node[above left,near end] {$\St_{j-1}$} (D);

  \fill[red] (X) circle (0.9pt);
  \node[above] at (X) {$x\in \mathcal{A}_0$};

  \coordinate (Xleft) at ($(X)+(-0.6,0)$);
  \pic[draw,->,"$\frac{\pi}{3}$", below left, angle radius=10.5pt,angle eccentricity=1.15]
     {angle = Xleft--X--D};
\end{tikzpicture}
\hfill
\begin{tikzpicture}[scale=2.6]
  \usetikzlibrary{calc}

  \draw[thick] (-0.6,0) -- (1.2,0);
  \node[above] at (-0.5,0.02) {$\omega$};
  \coordinate (X) at (0.35,0);
  \node[above left] at (X) {$x\in \mathcal{A}_0$};

  \draw[ultra thick, blue!70!black]
        (X) --++ (-60:0.25) coordinate (J) --++ (0:0.6) coordinate (H);
  \draw[thick, blue!70!black, dotted]
        (J) --++ (60:0.25) coordinate (F);
  \draw[ultra thick, blue!70!black]
        (J) --++ (-120:0.6) coordinate (T);

  \draw[blue!70!black] (0,-0.25) node[below] {$\St_{j}$};
  
  \fill[red] (X) circle (0.9pt);
  \fill[red] (F) circle (0.9pt);
  \fill[blue!70!black] (J) circle (0.9pt);
  \node[above right] at (F) {$x_j' \in P_j$};
  \node[below right] at (J) {$t(x)_j$};
 
  \draw[thick, blue!70!black, dotted]
        (J) --++ (0:1.3) coordinate (Z);
 
  \draw[thick,dotted] (1.2,0) to[out=0, in=150] (Z);
  \node[above right] at (Z) {$z(x)_j$};
  \fill[red] (Z) circle (0.9pt);
\end{tikzpicture}
\hfill~
    \caption{Transition from $\St_{j-1}$ to $\St_j$ in the neighborhood of vertices 
    $x = e^{\pi i/6}, e^{7\pi i/6}$ in the case of even $j$ and $x = e^{5\pi i/6}, e^{11\pi i/6}$ in the case of odd $j$}
    \label{pic:3stepI}
\end{figure}
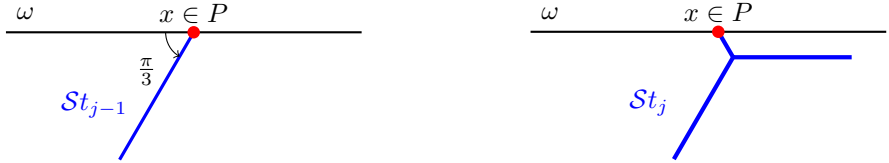

The combinatorial structure of the tree $\St_4$ is depicted in Fig.~\ref{pic:4topology}. However, the figure is not drawn to scale; as is evident from the rightmost part of Figure~\ref{pic:2step0}, even the transition from $\St_1$ to $\St_2$ cannot be depicted accurately.

The distance from $t(x)_j$ to $\omega$ is at most $\varepsilon_j$, and so
\[
|xz(x)_j| < 2|t(x)_jz(x)_j| < 2\sqrt{2\varepsilon_j}.
\]
Suppose that a tree $\St$ with topology $R$ not in $D(T_j)$ is a Steiner tree for $\mathcal{A}_j$. Then
\[
\H (\St) \leq \H(\St_j) < \H(\St_{j-1}) + 8\sqrt{2\varepsilon_j} < \H(\St_{j-1}) + \delta_{j-1}.
\]

Now, let us forget that vertices $P_j = \mathcal{A}_j \setminus \mathcal{A}_{j-1}$ are terminals of $\St$: if a vertex $p\in P_j$ has degree 1, then we delete $p$ from the topology $R$ (viewed as an abstract graph) and in the case that $p$ has degree two we contract the adjacent edges $\{p,v_1\}$, $\{p,v_2\}$ into $\{v_1,v_2\}$ and delete the vertex $p$. Let $\St'$ be the unique (by Proposition~\ref{pr:convex}) local minimizer for the reduced topology $R'$.
Since $\H(\St') \leq \H(\St) < \H(\St_{j-1}) + \delta_{j-1}$, the definition of $\delta_{j-1}$ implies $R' \in D(T_{j-1})$.
Thus the only possibility for $R \notin D(T_j)$ is that there exist point $p \in P_j$ that is adjacent to a point $q$ that is neither the corresponding point to $p_0 \in \mathcal{A}_0$ nor a common branching point with $p_0$. But in this case $|pq| > |pp_0|$, which is impossible in a Steiner tree. This is a contradiction.

By construction, the tree $\St_{j}$ is full, so by Proposition~\ref{pr:convex} the tree $\St_j$ is the unique solution to the Steiner problem for $\mathcal{A}_j$.
Thus, we can define $\delta_j > 0$\footnote{Calculations show that one can take $\delta_j =  const \cdot \varepsilon_j$, but we do not need the explicit form of the parameters.}.

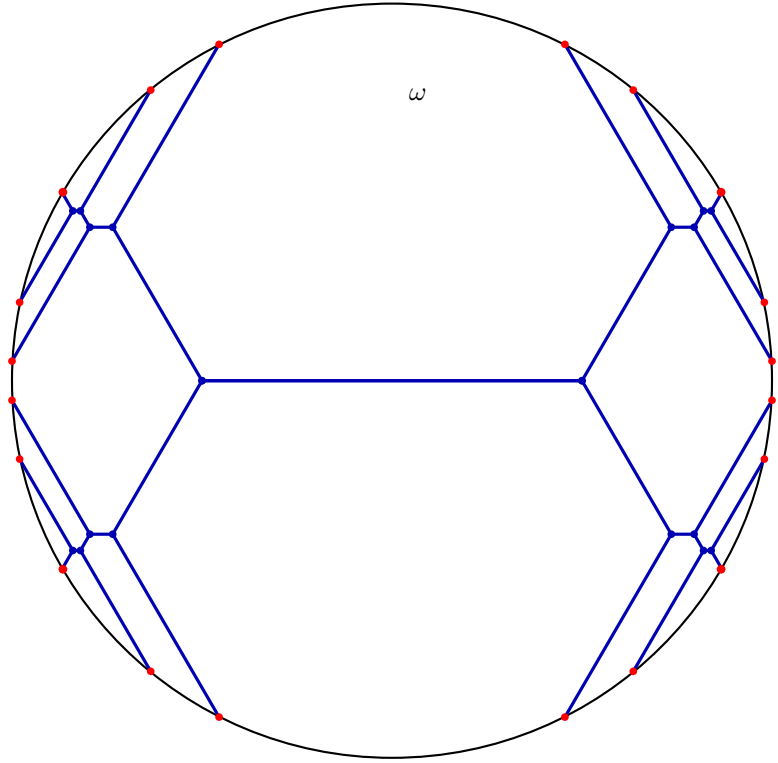
\begin{figure}[h]
    \centering
    \begin{tikzpicture}[scale=5]

  \pgfmathsetmacro{\rtthree}{1.7320508075688772} 
  \pgfmathsetmacro{\xC}{\rtthree/2}  
  \pgfmathsetmacro{\yC}{0.5}         
  \pgfmathsetmacro{\xs}{1/sqrt(3)}   

  \draw[thick] (0,0) circle (1);
  \node[below right] at (0.02,0.8) {$\omega$};

  \coordinate (UR) at  (\xC,\yC);     
  \coordinate (UL) at  (-\xC,\yC);    
  \coordinate (LL) at  (-\xC,-\yC);   
  \coordinate (LR) at  (\xC,-\yC);    

  \draw[very thick, blue!70!black] (0,0) coordinate(x0) --++ (0:0.5) coordinate (x1) --++ (60:0.47) coordinate (x2) --++ (120:0.56) coordinate (l1);
  \draw[very thick, blue!70!black] (x2) --++ (0:0.06) coordinate (x3) --++ (-60:0.41) coordinate (l2);
  \draw[very thick, blue!70!black] (x3) --++ (60:0.05) coordinate (x4) --++ (120:0.37) coordinate (l3);
  \draw[very thick, blue!70!black] (x4) --++ (0:0.02) coordinate (x5) --++ (-60:0.28) coordinate (l4);
  \draw[very thick, blue!70!black] (x5) --++ (60:0.065);
  
  \foreach \P in {l1,l2,l3,l4} \fill[red] (\P) circle (0.01012);
  \foreach \P in {x1,x2,x3,x4,x5} \fill[blue!70!black] (\P) circle (0.01012);

\begin{scope}[yscale=-1]
\draw[very thick, blue!70!black] (0,0) coordinate(x0) --++ (0:0.5) coordinate (x1) --++ (60:0.47) coordinate (x2) --++ (120:0.56) coordinate (l1);
  \draw[very thick, blue!70!black] (x2) --++ (0:0.06) coordinate (x3) --++ (-60:0.41) coordinate (l2);
  \draw[very thick, blue!70!black] (x3) --++ (60:0.05) coordinate (x4) --++ (120:0.37) coordinate (l3);
  \draw[very thick, blue!70!black] (x4) --++ (0:0.02) coordinate (x5) --++ (-60:0.28) coordinate (l4);
  \draw[very thick, blue!70!black] (x5) --++ (60:0.065);
  
  \foreach \P in {l1,l2,l3,l4} \fill[red] (\P) circle (0.01012);
  \foreach \P in {x1,x2,x3,x4,x5} \fill[blue!70!black] (\P) circle (0.01012);
\end{scope}

\begin{scope}[xscale=-1]
\draw[very thick, blue!70!black] (0,0) coordinate(x0) --++ (0:0.5) coordinate (x1) --++ (60:0.47) coordinate (x2) --++ (120:0.56) coordinate (l1);
  \draw[very thick, blue!70!black] (x2) --++ (0:0.06) coordinate (x3) --++ (-60:0.41) coordinate (l2);
  \draw[very thick, blue!70!black] (x3) --++ (60:0.05) coordinate (x4) --++ (120:0.37) coordinate (l3);
  \draw[very thick, blue!70!black] (x4) --++ (0:0.02) coordinate (x5) --++ (-60:0.28) coordinate (l4);
  \draw[very thick, blue!70!black] (x5) --++ (60:0.065);
  
  \foreach \P in {l1,l2,l3,l4} \fill[red] (\P) circle (0.01012);
  \foreach \P in {x1,x2,x3,x4,x5} \fill[blue!70!black] (\P) circle (0.01012);
\end{scope}

\begin{scope}[xscale=-1, yscale=-1]
\draw[very thick, blue!70!black] (0,0) coordinate(x0) --++ (0:0.5) coordinate (x1) --++ (60:0.47) coordinate (x2) --++ (120:0.56) coordinate (l1);
  \draw[very thick, blue!70!black] (x2) --++ (0:0.06) coordinate (x3) --++ (-60:0.41) coordinate (l2);
  \draw[very thick, blue!70!black] (x3) --++ (60:0.05) coordinate (x4) --++ (120:0.37) coordinate (l3);
  \draw[very thick, blue!70!black] (x4) --++ (0:0.02) coordinate (x5) --++ (-60:0.28) coordinate (l4);
  \draw[very thick, blue!70!black] (x5) --++ (60:0.065);
  
  \foreach \P in {l1,l2,l3,l4} \fill[red] (\P) circle (0.01012);
  \foreach \P in {x1,x2,x3,x4,x5} \fill[blue!70!black] (\P) circle (0.01012);
\end{scope}

  \foreach \P in {UL,UR,LL,LR} \fill[red] (\P) circle (0.0115);

\end{tikzpicture}
    \caption{A locally minimal tree with the same topology and symmetries as $\St_4$}
    \label{pic:4topology} 
\end{figure}

Finally, we can pass to the limit by setting $\mathcal{A}_\infty = \cup_{j=0}^\infty \mathcal{A}_j$. 
Note that a segment that appeared in $\St_{j}$ can only change (shorten) in $\St_{j+1}$, so we can define the limit tree $\St_\infty$, consisting of the segments that stabilize. By construction, $\St_\infty$ is a full tree of finite length connecting $\mathcal{A}_\infty$. Since $\mathcal{A}_j$ is contained in $\mathcal{A}_\infty$ we have
\[
\H (\St) \geq \H (\St_j)
\]
for any Steiner tree $\St$ with the terminal set $\mathcal{A}_\infty$. Passing to the limit, we get $\H (\St) \geq \H(\St_\infty)$, that is, $\St_\infty$ is a (possibly not unique) Steiner tree for $\mathcal{A}_\infty$.
\end{proof}

\section{Regularity} \label{sec:regularity}

\subsection{Lower bound}

We claim that for the vertices of the hypercube $H = \{-1,1\}^d \subset \mathbb{R}^d$, a Steiner tree is necessarily full.
By Lemma~\ref{lm:conv}, a Steiner tree $\St$ for $H$ lies in the convex hull of $H$, and for any points $x \in H$, $y,z \in \conv H$ the angle $\angle yxz$ is non-obtuse. 
Hence $x$ cannot have degree 2 or 3 in $\St$.
    
The length of a minimum spanning tree for $H$ is $2 (2^d-1)$, so by Proposition~\ref{pr:ratio}
\[
\H(\St) \geq \frac{2}{\sqrt{3}} \left (2^d-1 \right),
\]
where $\St$ is any Steiner tree for $H$. The hypercube $H$ is inscribed in a sphere of radius $\sqrt{d}$, which gives the estimate
\[
\H(\St_r) = \H(\St_r \cap B_1(o)) \geq \frac{2(2^d-1)}{\sqrt{3d}},
\]
where $\St_r$ is a Steiner tree for a rescaled hypercube, inscribed in $B_1(o)$ (so the scale factor is $1/\sqrt{d}$).

Now let us show that a main part of this example belongs to a smaller ball $B_\rho(o)$, where $\rho$ is close to~1.

\begin{lemma} \label{lm:lowerbound}
   Let $\St$ be a Steiner tree for $A = \left \{\frac{-1}{\sqrt{d}}, \frac{1}{\sqrt{d}} \right \}^d$. 
   Then 
   \[
   H^1(\St \cap B_{\rho}(o)) \geq \frac{2(2^d-1)}{\sqrt{3d}} \cdot \left (1 - 2\sqrt{3d(1-\rho^2)} \right),
   \]
   provided that $1 \geq \rho \geq 1 - \frac{1}{10d}$.
\end{lemma}

\begin{proof}
    Let $S$ be a connected component of the set $\left [\frac{-1}{\sqrt{d}}, \frac{1}{\sqrt{d}} \right]^d \setminus B_{\rho}(o)$.
    Clearly, $S$ consists of points with the same signs of coordinates (so there are $2^d$ such components); 
    indeed, any point with at least one zero coordinate inside the hypercube $[\frac{-1}{\sqrt{d}}, \frac{1}{\sqrt{d}}]^d$ has distance at most
    \[
    \sqrt{0 + 1/d + \dots + 1/d} = \sqrt{\frac{d-1}{d}} < 1 - \frac{1}{10d}
    \]
    from $o$ and thus belongs to $B_\rho(o)$.

    Without loss of generality, let $S$ contains point $v$ with all coordinates equal to $1/\sqrt{d}$. Then the distance from $v$ to any point from the set 
    $\partial S \cap \partial B_\rho (o)$ is at most
    \[
    \sqrt{\left (\frac{1}{\sqrt{d}}-x_1 \right)^2+\dots + \left (\frac{1}{\sqrt{d}}-x_d \right )^2} \leq \sqrt{1 - x_1^2 - x_2^2 -\dots - x_d^2} = \sqrt{1 - \rho^2},
    \]
    where we use $0 \leq x_i \leq 1/\sqrt{d}$ and $(a-b)^2 \leq a^2 - b^2$ for $0 < b < a$.

    Now recall that $\St$ is a full tree and thus it consists of $2 \cdot 2^d-3$ line segments. Since a segment intersects a sphere in at most two points, we have
    \[
    \# (\St \cap \partial B_\rho(o)) \leq 4 \cdot (2^d - 1.5).
    \]
    Apply Lemma~\ref{lm:lengthbound} with $X = \St \setminus B_\rho(o)$ and $Y$ being the union of all segments connecting points $\St \cap \partial     B_\rho(o)$ with a nearest vertex of a hypercube. The previous bounds give 
    \[
    \H(X) \leq \H(Y) \leq 4 (2^d - 1.5) \sqrt{1 - \rho^2} < 4 (2^d - 1) \sqrt{1 - \rho^2}.
    \]
    Thus
    \[
    \H(\St \cap B_\rho(o)) = \H(\St) - \H(X) \geq \frac{2(2^d-1)}{\sqrt{3d}} - 4 \cdot (2^d - 1) \sqrt{1 - \rho^2} \geq \frac{2(2^d-1)}{\sqrt{3d}} \left ( 1 - 2\sqrt{3d(1-\rho^2)}\right).
    \]
\end{proof}

\subsection{Upper bound} \label{subsec:mainproofs}

\begin{proof}[Proof of Theorem~\ref{th:main}]
    
By rescaling, put $s = 1$.
Define 
\[
t_r :=  \# (\partial B_r(x) \cap \St),
\]
\[
L_r := \mathcal{H}^1(B_r(x) \cap \St).
\]
By Corollary~\ref{cor:everyeps}, for every $r<1$ the set $\St \cap B_r(x)$ consists of a finite number of segments, so $t_r$ is finite and the quantity $L_r$ is continuous as a function of $r$ for $r \in [0,1)$.

By Theorem~\ref{th:tpoints} there exists a set $S$ connecting $\partial B_r(x) \cap \St$ with length at most $12r t_r^\frac{d-2}{d-1}$, so by Lemma~\ref{lm:lengthbound} for $X = B_r(x) \cap \St$ and $Y = S$, we have
\[
L_r \leq 12r t_r^\frac{d-2}{d-1},
\]
because $\St \setminus (B_r(x) \cap \St) \cup S$ is connected.
Thus
\[
t_r \geq  \frac{L_r^{\frac{d-1}{d-2}}}{144}
\]
for every $r \in [\rho,1)$. Lemma~\ref{lm:coarea} implies 
\begin{equation} \label{ineq:main}
L_R \geq L_\rho + \int_\rho^R t_r \,dr 
   \geq L_\rho + \frac{1}{144} \int_\rho^R L_r^{\frac{d-1}{d-2}} dr    
\end{equation}

Let $f(r)$ be the smallest function satisfying the inequality~\eqref{ineq:main}. Then
\[
f(z) = a + b \int_\rho^z f(y)^\alpha \, dy, 
\quad a = L_\rho, \quad b = \frac{1}{144}, \quad \alpha = 1 + \frac{1}{d-2}.
\]
Taking the derivative we obtain
\[
f'(z) = b \cdot f(z)^\alpha, \quad f(\rho) = a.
\]
Let us solve this differential equation:
\[
\frac{df}{dz} = b f^\alpha, \quad f^{-\alpha} df = b dz, \quad \frac{f^{1-\alpha}}{1-\alpha} = bz + C.
\]
So,
\[
f = \left ( C - \frac{bz}{d-2} \right )^{-(d-2)}
\]
and
\[
f(\rho) = a =  L_\rho = \left (C - \frac{b\rho}{d-2}  \right )^{-(d-2)}
\]
Thus,
\[
C = L_\rho^{-\tfrac{1}{d-2}} + \frac{b\rho}{d-2}
\]

If $C = \frac{bz}{d-2}$ for some $z \in (\rho,1)$ then $f$ and the length $L_R$ are unbounded, which is a contradiction.
Thus 
\[
z = \frac{C(d-2)}{b} \geq 1,
\]
which is equivalent to
\[
L_\rho \leq \left ( \frac{144(d-2)}{1-\rho} \right) ^{d-2}.
\]
After rescaling we obtain the desired bound.

\end{proof}

\begin{proof}[Proof of Corollary~\ref{cor:main}]
   Again, by rescaling, put $s = 1$ and choose $q$ in the interval $(\rho,1)$. Let $t$ be a minimum value of 
   \[
   \# ( \partial B_r(x) \cap \St)
   \]
   for $r \in [\rho,q]$; it exists because by Corollary~\ref{cor:everyeps} the set $\overline{B_q(x) \cap \St}$ consists of a finite number of line segments.
   By Lemma~\ref{lm:coarea} and Theorem~\ref{th:main} 
   \[
    (q-\rho) t \leq  \H (B_q(x) \cap \St) \leq   \left ( \frac{144d}{1-q} \right) ^{d-2},
   \]
   so
   \[
   t \leq \frac{144d^{d-2}}{(q-\rho)(1-q)^{d-2}}.
   \]
   Optimization over $q$ gives $q := \frac{1 + (d-2)\rho}{d-1}$ which implies
   \[
   t \leq \frac{144d^{d-2}(d-1)^{d-1}}{(1-\rho)^{d-1}(d-2)^{d-2}} < \frac{1}{2} \frac{144d^{d-1}}{(1-\rho)^{d-1}}. 
   \]
Let $u \in [\rho,q]$ be a radius for which $\# (\partial B_u(x) \cap \St) = t$. 
   The Steiner set $\St$ has no terminals inside $B_u(x)$, so the number of segments in the forest $\St \cap B_u(x)$ is at most
   \[
   2t-3 \leq \frac{144d^{d-1}}{(1-\rho)^{d-1}}.
   \]
   Convexity of the ball $B_\rho(x)$ implies that every such segment intersects $B_\rho(x)$ by a (possibly empty) segment.
\end{proof}

\begin{proof}[Proof of Corollary~\ref{cor:inNotationofPaoSte}]
  Since $x \in \St_\varepsilon$ we have
  \[
  B_\varepsilon (x) \cap \mathcal{A} = \emptyset,
  \]
  so we may apply Theorem~\ref{th:main} to $\St$, $x$ and $s = \varepsilon$.
  Then by the inclusion $\St_\varepsilon \subset \St$ we have
  \[
  \H (\St_\varepsilon \cap B_{\rho \varepsilon}(x) ) \leq \H (\St \cap B_{\rho \varepsilon}(x) )  \leq \varepsilon \left ( \frac{144d}{1-\rho} \right) ^{d-2}.
  \]
  The same inclusion allows us to use the bound from Corollary~\ref{cor:main} for the set $\St_\varepsilon \cap B_{\rho \varepsilon}(x)$.
\end{proof}

\subsection{Planar case} \label{subsec:planar}

The following theorem shows that the example from Section 3 is, in some sense, the worst possible in the plane.

\begin{theorem} \label{th:4pointsmax}
    Let $\mathcal{A}$ be a totally disconnected set lying on a circle $\omega \subset \mathbb{R}^2$.
    Suppose that a solution $\St$ to the Steiner problem for $\mathcal{A}$ is a full tree.
    Then $\mathcal{A}$ has at most 4 accumulation points.
\end{theorem}

Obviously, Theorem~\ref{th:4pointsmax} implies that the set of branching points of $\St$ has at most (the same) 4 accumulation points.

\begin{proof}
   Without loss of generality, let $\omega$ be a unit circle centered at the origin. 
   Since $\St$ is a full tree, the edges of $\St$ have only 3 possible directions; let us rotate the picture so that one of these directions becomes the real axis.

   Now consider $\St_\varepsilon = \St \cap \overline{B_{1-\varepsilon} (o)}$, where $o$ is the origin and $\varepsilon > 0$.
   By Corollary~\ref{cor:everyeps} $\St_\varepsilon$ is a finite tree; let $\{p_1, \dots, p_n\} = \St \cap \partial B_{1-\varepsilon}(o)$. 
   Corollary~\ref{cor:maxwellCocircular} and rescaling give that $\H(\St_\varepsilon)$ is the sum of $(1-\varepsilon)\cos \alpha_k$, where $\alpha_k$ is the angle between the radius $[op_k]$ of $B_{1-\varepsilon}(o)$ and $\St_\varepsilon$ at the point $p_k$. Clearly, $\alpha_k$ belongs to $[0,\pi/2)$, so every summand is positive.

   Now consider an accumulation point $x \in \mathcal A$. If $x^6 \neq -1$ (here $x$ is considered as a complex number), then for a small enough $\delta > 0$ there exists $q_\delta > 0$ such that for every $y \in B_\delta(x)$ one has $\cos \alpha > q_\delta$, 
   where $\alpha$ is an angle between a segment with a possible direction of $\St$ passing through $y$ and the radius $[oy]$ of the circle $\partial B_{|y|}(o)$ containing $y$ and centered at the origin. 
   By Corollary~\ref{cor:cocircularlength}(i), $\H(\St) \leq 2\pi$.
   This means that for every $\varepsilon > 0$ the set $\St_\varepsilon$ has at most $2\pi / q_\delta$ leaves inside $B_\delta(x)$, so $x$ can be an accumulation point only if $x^6 = -1$.

   To reduce the upper bound from 6 to 4 we need the following lemma.

   \begin{lemma} \label{lm:emelevskaya}
       If $x \in \omega$ is an accumulation point of $\mathcal A$, then it is an accumulation point from both sides.
       Formally, for every $\varepsilon > 0$ every connected component of $(\omega \cap B_\varepsilon (x) \setminus \{x\})$ has a nonempty intersection with $\mathcal{A}$.
   \end{lemma}

   \begin{proof}
       Recall that $x^6 = -1$ provided that $\St$ contains a segment parallel to the real axis, so without loss of generality let $x = e^{\pi i/2}$.
       Consider a connected component $S$ of $\St \cap R(1,1-\varepsilon)$, where $R(1,1-\varepsilon)$ is the closed ring between $\omega = B_1(o)$ and $B_{1-\varepsilon}(o)$, such that $S$ has an infinite number of points at $\omega$. Such a component exists since $\St_\varepsilon$ is a finite graph, $\mathcal{A}$ is infinite and $\St$ is full. 

       Then consider a connected component $S_\delta$ of the intersection of the ring $R(1-\delta,1-\varepsilon)$ with $S$ for some $\delta < \varepsilon$. 
       Corollary~\ref{cor:windrose} applied to $S_\delta$ gives that the absolute value of the sum of $\bar{c_k}$ over vertices from $\partial B_{1-\delta}(o)$ is bounded by $\# (\partial B_{1-\varepsilon}(o) \cap S)$. Note that the coefficients $\bar{c_k}$ that correspond to the terminals belonging to a sufficiently small $\rho$-neighborhood of $x$ have to lie in the consecutive subset $\{\pm 1, e^{\pi i/3}, e^{2\pi i/3}\}$ of the possible set of 6-th roots of unity.
       Since the sum is bounded, all the summands, except a finite number $C_\varepsilon$ are either $-1$ or 1, which are produced by vertices from different arcs of 
       $(B_{1-\delta}(o) \cap B_\rho (x) \setminus \{x\})$. But as $\delta$ tends to 0, the number of summands of both types tends to infinity, so there are infinitely many points on both arcs of $(\omega \cap B_\rho (x) \setminus \{x\})$. 
   \end{proof}
   
   Suppose that $\mathcal{A}$ has at least 5 accumulation points. Then they must be five consecutive points, satisfying $x^6 = -1$. By Lemma~\ref{lm:emelevskaya} at most one component of $\omega \setminus \mathcal{A}$ has diameter more than $1 - \delta$ for some positive $\delta$. 
   Corollary~\ref{cor:5pointsaremany} completes the proof by contradiction with the fact that $\St$ is full.  

\end{proof}

\begin{remark}
    We cannot avoid this epsilondeltamanship because there is no analog of Corollary~\ref{cor:windrose} for infinite trees. Indeed, an infinite sum of numbers on the complex unit circle does not necessarily converge. Also, in the proof of Lemma~\ref{lm:Maxwell} it is possible to replace the summation over edges by the summation over vertices only if the sum of vertices (as complex numbers) converges absolutely. Again, this is not the case for terminals located on the complex unit circle.
\end{remark}

\begin{lemma} \label{lm:squarerootoftree}
    The length of a full Steiner tree with at least three terminals on a unit circle is at least $\sqrt{3}$.
\end{lemma}

\begin{proof}
    Suppose the contrary. Then there is a full Steiner tree $\St$ with the length at most $\sqrt{3}(1 - \varepsilon)$, which means that a finite full tree $S_\varepsilon$ inside $B_{1-\varepsilon}(o)$ also has length at most $\sqrt{3}(1 - \varepsilon)$.
    We may assume that $S_\varepsilon$ has at least 3 terminals, otherwise take smaller $\varepsilon$.
    
    So $S_\varepsilon$ is a locally minimal tree and one can apply Lemma~\ref{lm:Melzak} (Melzak reduction).
    By the pigeon-hole principle there are at least two choices of $p_1,p_2$ and $q$, where $p_1$ and $p_2$ are terminals adjacent to the branching point $q$. Since $S_\varepsilon$ is a tree exactly one of the arcs $\breve{p_1p_2}$ does not contain any other terminals. The interiors of these arcs for two choices are disjoint, so one of them has angular measure smaller than $4\pi/3$, and so the Melzak point $p$ lies outside the circle $\partial B_{1-\varepsilon}(o)$. Also, this arc $\breve{p_1p_2}$ belongs to the quadrilateral $pp_1qp_2$ and so Melzak reduction in this case produces a tree $S'_\varepsilon$. 
    Note that $S'_\varepsilon \cap B_{1-\varepsilon}(o)$ is a full tree with the length at most $\sqrt{3}(1 - \varepsilon)$, so we can repeat the reduction until we obtain a tree with 3 terminals $a_1$, $a_2$ and $a_3$.
    If the diameter of $\{a_1,a_2,a_3\}$ is at most $\sqrt{3}(1 - \varepsilon)$ then they can be covered by an arc $R$ of the angular measure $2\pi/3$; this gives that one of the angles, say $\angle a_1a_2a_3$, is at least $2\pi/3$.    
    This contradicts the existence of a full tree (in this case it is a regular tripod) with terminals $a_1$, $a_2$ and $a_3$.
    Finally, the length of any connected set containing $\{a_1,a_2,a_3\}$ is at least $\diam \{a_1,a_2,a_3\}$.
\end{proof}

\begin{corollary} \label{cor:2components}
    Let $\St$ be a Steiner tree for $\mathcal{A} \subset \mathbb{R}^2$ with $\H(\St) < \infty$.
    Assume that for some $s > 0$, the open ball $B_s(x)$ has empty intersection with $\mathcal{A}$. 
    Then at most two connected components of $\St \cap B_s(x)$ have branching points.
\end{corollary}

\begin{proof}
   By Lemma~\ref{lm:squarerootoftree}, the presence of three such connected components gives the length of their union to be at least $3\sqrt{3} = 5.1916\dots$, which contradicts the bound from Corollary~\ref{cor:cocircularlength}(ii).
\end{proof}

In particular, Corollary~\ref{cor:2components} and Theorem~\ref{th:4pointsmax} imply that in the plane the set of branching points of $\St \cap B_s(x)$ has at most 8 accumulation points, provided that the terminal set $\mathcal{A}$ is disjoint from $B_s(x)$. This bound does not seem to be tight.

\section{Open questions} \label{sec:open}

Various questions on Steiner trees are collected in~\cite{teplitskaya2025open}.

We want to understand the behavior of the upper bounds in Theorem~\ref{th:main} and Corollary~\ref{cor:main} with respect to $d$ and $1-\rho$.
There is a large gap between the mentioned upper bounds and the lower bounds in the spirit of Lemma~\ref{lm:lowerbound}. 
One of the obstacles is the lack of nontrivial examples in high dimensions, see the discussion in~\cite{fleischmann2025steiner}.

Also, note that the results by Paolini and Stepanov work in a proper metric space $\mathcal X$. 
The quantitative picture in various metric spaces is expected to be diverse, as we can see even in the planar and non-planar Euclidean cases.

Quantitative regularity of other one-dimensional problems is also of interest.
One of the challenging examples is Branched Optimal Transport (Gilbert--Steiner problem).
The Gilbert--Steiner problem is a generalization of the Steiner tree problem and specific optimal mass transportation, which allows the use of additional (branching) point in a transport plan. A specific feature of the problem is that the cost of transporting a mass $m$ along a segment of length $l$ is equal to $m^p \cdot l$ for a fixed $0 \leq p < 1$ and segments may end at points not belonging to the supports of given measures (branching points). 
Our arguments in Theorem~\ref{th:main} and Corollary~\ref{cor:main} can be repeated word for word for the case $p=0$.
For more details, see~\cite{bernot2009optimal}.

\paragraph{Acknowledgments.} The authors are grateful to Ilya Bogdanov for his questions, which led to clarification of the statements.

\paragraph{Funding.} This work started at Summer Scientific Workshop ``Introduction to Modern Mathematics and Related Scientific Projects'' supported by the Bulgarian Ministry of Education and Science, Scientific Programme ``Enhancing the Research Capacity in Mathematical Sciences (PIKOM)'', No. D01-88/20.06.2025.
The work of P. Prozorov was supported by the Russian Science Foundation (project no. 25-11-00058). 
Y. Teplitskaya is supported by the French National Research Agency (ANR) under grant ANR-21-CE40-0013-01 (project GeMfaceT).

\paragraph{Data availability statement.} This manuscript has no associated data.

\paragraph{Conflict of Interest.} The authors have no relevant financial or non-financial interests to disclose.

\bibliography{main}
\bibliographystyle{plain}

\end{document}